\documentclass[12pt]{amsart}

\usepackage{a4wide}
\usepackage[T1]{fontenc}
\usepackage[utf8]{inputenc}
\usepackage{setspace}

\usepackage{amsmath, amsfonts, amsthm}
\usepackage{mathtools}
\usepackage{color}
\usepackage{float}
\usepackage{url}

\usepackage{array}
\newcolumntype{P}[1]{>{\centering\arraybackslash}p{#1}}
\usepackage{caption}
\newtheorem{lemma}{Lemma}

\newtheorem{theorem}[lemma]{Theorem}
\newtheorem{corollary}[lemma]{Corollary}
\theoremstyle{remark}

\newcommand{\E}{\mathbb{E}}
\newcommand{\V}{\mathbb{V}}
\restylefloat{figure}
\makeatletter
\let\@mkboth\@gobbletwo
\let\@oddhead\@empty
\let\@evenhead\@empty
\makeatother
\usepackage{blindtext}
\usepackage{fancyhdr}
\usepackage{blindtext}
\usepackage{lastpage}
\usepackage{fancyhdr}
\title{Tree height and the asymptotic mean \\ of the Colijn--Plazzotta rank \\ of unlabeled binary rooted trees}

\author{Luc Devroye, Michael R.~Doboli, Noah A.~Rosenberg and Stephan Wagner}

\date\today

\begin{document}

\begin{abstract}
The Colijn--Plazzotta ranking is a bijective encoding of the unlabeled binary rooted trees with positive integers. We show that the rank $f(t)$ of a tree $t$ is closely related to its height $h$, the length of the longest path from a leaf to the root. We consider the rank $f(\tau_n)$ of a random $n$-leaf tree $\tau_n$ under each of three models: (i) uniformly random unlabeled unordered binary rooted trees, or unlabeled topologies; (ii) uniformly random leaf-labeled binary trees, or labeled topologies under the uniform model; and (iii) random binary search trees, or labeled topologies under the Yule--Harding model. Relying on the close relationship between tree rank and tree height, we obtain results concerning the asymptotic properties of $\log \log f(\tau_n)$. In particular, we find $\E \{\log_2 \log f(\tau_n)\} \sim 2 \sqrt{\pi n}$ for uniformly random unlabeled ordered binary rooted trees and uniformly random leaf-labeled binary trees, and for a constant $\alpha \approx 4.31107$, $\E\{\log_2 \log f(\tau_n)\} \sim \alpha \log n $ for leaf-labeled binary trees under the Yule--Harding model. We show that the mean of $f(\tau_n)$ itself under the three models is largely determined by the rank $c_{n-1}$ of the highest-ranked tree---the caterpillar---obtaining an asymptotic relationship with $\pi_n c_{n-1}$, where $\pi_n$ is a model-specific function of $n$. The results resolve open problems, providing a new class of results on an encoding useful in mathematical phylogenetics.
\end{abstract}


\maketitle
\footnotetext[1]
{Luc Devroye (lucdevroye@gmail.com) is with the School of Computer Science at McGill University in Montreal, Canada.}
\footnotetext[2]
{Michael R.~Doboli (mdoboli@stanford.edu) and Noah A.~Rosenberg (noahr@stanford.edu) are with the Department of Biology at Stanford University, California.}
\footnotetext[3]
{Stephan Wagner (stephan.wagner@math.uu.se) is with the Institute of Discrete Mathematics, TU Graz, Austria, and the Department of Mathematics, Uppsala University, Sweden.}

\section{Introduction}

The Colijn--Plazzotta rank $f(t)$ of a binary rooted tree $t$ is defined recursively as follows: if $\ell(t)$ and $r(t)$ are the left and right subtree, respectively, arranged in such a way that $f\big(\ell(t)\big) \geq f\big(r(t)\big)$, then
\[f(t) = \frac{f\big(\ell(t)\big) \, \big(f(\ell(t))-1\big)}{2} + 1 + f\big(r(t)\big).\]
The rank $1$ is assigned to a tree with a single leaf. 

Colijn--Plazzotta rank, or \emph{CP rank}, has been introduced as a tree statistic from the context of mathematical evolutionary biology. In the study of evolutionary trees, statistical summaries of trees are often used for characterizing the outcomes of evolutionary models and for statistical inference of the processes that have given rise to the trees~\cite{FischerEtAl2023}. CP rank has been used as a summary of tree shape in empirical scenarios in which trees of biological relationships are unconcerned with leaf labels, such as in examples with trees of sequences from the same pathogenic organism~\cite{Colijn+Plazzotta2018}.

Informally, for a fixed number of leaves, the CP rank is lowest for balanced trees and greatest for unbalanced trees. It has therefore been proposed as a measure of tree balance~\cite{FischerEtAl2023, Rosenberg2021}. In a compilation of mathematical results for tree balance indices that capture many different features of rooted trees, Fischer et al.~\cite{FischerEtAl2023} have listed a set of basic properties that are of interest for any balance index. Among these are the minimal and maximal values of the index across all trees with a fixed number of leaves, and the mean and variance of the index under the two most frequently used probabilistic models in mathematical phylogenetics: the uniform model, which assigns equal probability to all binary rooted labeled trees with a fixed number of leaves, and the Yule--Harding model, in which, conditional on the number of leaves, the probability of a binary rooted labeled tree is proportional to the number of sequences of bifurcations that can give rise to the tree. The mathematical properties of balance indices assist in characterizing the way that balance indices relate to one another and how they perform in empirical settings.

The trees of minimal and maximal CP rank for a fixed number of leaves have been characterized~\cite{Rosenberg2021}, and indeed the asymptotic CP ranks of these trees in terms of the number of leaves have also been obtained~\cite{DoboliEtAl2024, Rosenberg2021}. The mean and variance under the uniform and Yule--Harding models have been listed as open problems~\cite[p.~243]{FischerEtAl2023}.

We show here that the asymptotic mean and variance under the Yule--Harding model can be obtained by a connection between this model in the phylogenetics setting and the nearly equivalent formulation of random binary search trees in computer science. First, we show that the order of magnitude of the CP rank of a tree is determined by the height of the tree, the greatest distance from the root to a leaf. By connecting the CP rank to tree height and in turn to probabilistic results for the height, we obtain distributional properties of the CP rank under the Yule--Harding model. We also obtain related results on the closely related uniform model on labeled binary rooted trees and the uniform model on \emph{un}labeled binary rooted trees.

\bigskip
\section{Tree height and the Colijn--Plazzotta rank}

We consider all trees to be binary and rooted. The \emph{height} of a tree is the length in edges of the longest path from the root to a leaf. Two special families of binary trees with $n$ leaves play a key role in our analysis: the caterpillars, and the pseudocaterpillars. In a \emph{caterpillar} with $n$ leaves, $n \geq 1$, every non-leaf has at least one leaf child. This condition forces each caterpillar to consist of a chain of $n-1$ internal (i.e.~non-leaf) nodes to which a layer of external nodes is added. The \emph{pseudocaterpillars}~\cite{Rosenberg2007} (or \emph{4-pseudocaterpillars} in the terminology of \cite{Alimpiev2021}) can be constructed as follows for $n \geq 4$: start with a chain of $n-3$ internal nodes. Give the bottom node in the chain two children, and finally, complete the tree by adding a layer of $n$ external nodes. Caterpillars have height $n-1$, and pseudocaterpillars have height $n-2$.

Among unlabeled trees with a fixed number of leaves, Rosenberg~\cite[Corollary 10]{Rosenberg2021} found that the tree with the largest CP rank was the caterpillar. The CP rank of the caterpillar tree with $n$ leaves can be computed recursively by a sequence $c_n$~\cite[Theorem 9]{Rosenberg2021}. It is convenient to shift the index of the sequence by $1$ so that $c_k$ instead corresponds to the CP rank of the caterpillar with height $k$ and $k+1$ leaves. The sequence $c_k$ begins 1, 2, 3, 5, 12, 68, 2280 starting at $k=0$ (OEIS A108225 \cite{OEIS}). 

\begin{lemma}\label{rank-height}
Let the sequence $c_k$ be defined by $c_0 = 1$ and $c_{k+1} = c_{k}(c_{k}-1)/2 + 2$ for $k \geq 0$. For every tree $t$ of height $h$, we have
\[c_h \leq f(t) < c_{h+1}.\]
\end{lemma}

\begin{proof}
The proof proceeds by induction on $h$. For $h=0$, the tree consists of a single leaf, and we have $1 = c_0 =  f(t) < c_1 = 2$. Thus, the statement holds in this case, and we can proceed with the induction step. 

For a tree $t$ of height $h$, suppose $h_{\ell} < h$ and $h_r < h$ are the heights of subtrees $\ell(t)$ and $r(t)$, respectively. From the induction hypothesis for trees of height less than $h$ and the left--right arrangement so that $f\big( \ell(t) \big) \geq f\big(r(t) \big)$, it follows that
\[c_{h_r} \leq f(r(t)) \leq f(\ell(t)) < c_{h_{\ell}+1}.\]
The sequence $c_k$ is increasing~\cite[Lemma 8]{Rosenberg2021}, so that $h_r < h_\ell + 1$, and hence, $h_r \leq h_{\ell}$. 

Because $h = \max(h_{\ell},h_r) + 1$, it follows that $h_{\ell} = h - 1$. Thus, we have, again by the induction hypothesis,
\begin{align*}
f(t) &= \frac{f\big(\ell(t)\big) \, \big(f\big(\ell(t)\big)-1\big)}{2} + 1 + f\big(r(t)\big) \\
&\geq \frac{f\big(\ell(t)\big) \, \big(f\big(\ell(t)\big)-1\big)}{2} + 1 + 1 \\
&\geq \frac{c_{h-1}(c_{h-1}-1)}{2} + 2 \\
&= c_h,
\end{align*} 
which proves the lower bound. On the other hand,
\begin{align*}
f(t) &= \frac{f\big(\ell(t)\big) \, \big(f(\ell(t))-1 \big)}{2} + 1 + f\big(r(t)\big) \\
&\leq \frac{f\big(\ell(t)\big) \, \big(f(\ell(t))-1\big)}{2} + 1 + f\big(\ell(t)\big) \\
&= \frac{f\big(\ell(t)\big) \, \big(f\big(\ell(t)\big)+1\big)}{2} + 1 \\
&\leq \frac{(c_h-1)c_h}{2} + 1 \\
&= c_{h+1} - 1,
\end{align*} 
 proving the upper bound. This completes the induction.
\end{proof}

We conclude that the behavior of the height is to a great extent responsible for the behavior of the Colijn--Plazzotta rank of a tree. Indeed, because the CP rank is bijective with the positive integers, the lemma implies that as the positive integers are traversed, for each $h\geq 0$, the ranking proceeds through trees with height $h$, then proceeds to those with height $h+1$, and so on. We immediately obtain the following corollaries (which are well known, see~\cite{HPR1992}).
\begin{corollary}
For $h\geq 0$, the number of unlabeled binary rooted trees with height at most $h$ is $c_{h+1}-1$.
\end{corollary}
\begin{corollary}
For $h\geq 0$, the number of unlabeled binary rooted trees with height exactly $h$ is $c_{h+1}-c_h$.
\end{corollary}
The sequence $c_{h+1}-1$ begins at $h=0$ with values 1, 2, 4, 11, 67, 2279 (OEIS A006894 \cite{OEIS}). The sequence $c_{h+1}-c_h$ begins at $h=0$ with values 1, 1, 2, 7, 56, 2212 (OEIS A002658 \cite{OEIS}). 

According to~\cite[Corollary 14]{Rosenberg2021}, $c_k \sim 2 \gamma^{2^k}$ for a constant $\gamma \approx 1.11625$ as $k \to \infty$; note that $\gamma = \beta^2$ for the constant $\beta$ in~\cite{Rosenberg2021}, owing to the shift by 1 in $c_k$ relative to the indexing in~\cite{Rosenberg2021}. We immediately obtain the following result.
\begin{corollary}\label{cor:CP_vs_height}
For a tree $t$ with height $h$, we have
\[2^h + O(1) \leq \log_{\gamma} f(t) \leq 2^{h+1} + O(1),\]
and thus
\[\log_2 \log_{\gamma} f(t) = \log_2 \log f(t) + O(1) = h + O(1).\]
\end{corollary}
We now analyze the behavior of the CP rank of random trees, which is mainly determined by the height. Indeed, we proceed by making use of extensive probabilistic results available on tree height under different sets of assumptions. 
\bigskip
\section{Uniformly random unlabeled binary trees}
\label{sec:unlabeled}

Consider an unlabeled binary rooted tree on $n$ leaves, where each internal node has two children. A distinction exists between trees in which the left-right order of the children matters (\emph{ordered binary trees}) and those in which the order is irrelevant (\emph{unordered binary trees}, or \emph{unlabeled topologies} in the terminology of mathematical phylogenetics, or \emph{Otter trees} after \cite{Otter1948}). 

Let $\tau_n$ be a uniformly random ordered binary tree, also called a random \emph{Catalan tree} because the number of such trees on $n \geq 1$ leaves is 
$$
K_{n-1} = \frac{1}{n} \binom{2n-2}{n-1},
$$
where $K_n$ is the $n$-th Catalan number~\cite[Exercise 5]{Stanley2015}. A tree under this model of randomness is generated by considering a uniformly random ordered binary tree on $n-1$ (internal) nodes and adding a layer of external nodes. The random Catalan tree has been thoroughly studied~\cite[p.~224]{Sedgewick+Flajolet1996}. We denote the CP rank of a random Catalan tree by $C_n$ ($C$ for \underline{C}atalan).

Let $\tau'_n$ be a uniformly random unordered binary tree, a uniformly random Otter tree. The number of such trees can be calculated recursively. The exact value $U_n$ for the number of such trees on $n$ leaves follows
\begin{align}
\label{eq:Wedderburn}
U_n = \begin{cases}
    1, & n=1, \\
    \sum_{j=1}^{(n-1)/2} U_j U_{n-j}, & \text{odd } n \geq 3, \\
    \bigg( \sum_{j=1}^{n/2 - 1} U_j U_{n-j} \bigg) + \frac{U_{n/2} (U_{n/2} + 1)}{2}, & \text{even } n \geq 2. 
\end{cases}
\end{align}
The asymptotic approximation follows~\cite{Harding1971, Otter1948} 
\begin{align}\label{ottercount}
  U_n \sim \big(1+o(1)\big) \frac{\lambda}{n^{3/2} \rho^n},   
\end{align} 
where $\lambda \approx 0.31878$ and $\rho \approx 0.40270$.
The CP rank of a random Otter tree is denoted by $O_n$ ($O$ for \underline{O}tter).

To understand Theorem \ref{uniformtree}, we define a \emph{theta random variable} as a random variable with distribution function~\cite{Devroye1997}
\begin{equation}\label{eq:theta}
F(x) =
  \frac{4 \pi^{5/2}}{x^3} \sum_{j=1}^\infty j^2
      e^{-\pi^2 j^2/x^2} 
      =
  \sum_{j=-\infty}^\infty (1-2j^2x^2) e^{-j^2 x^2}~, x > 0.
\end{equation}
CP rank is defined for unordered binary trees. To extend the CP rank to ordered binary trees, we compute the CP rank of the unordered binary tree associated with an ordered binary tree.

\begin{theorem}\label{uniformtree}
(i) Let $\tau_n$ be a uniformly random unlabeled ordered binary tree with $n$ leaves, with CP rank $C_n=f(\tau_n)$. Then
$$
\mathbb{E} \{ \log_2 \log C_n \} \sim 2 \sqrt{\pi n},
$$
and 
$$
\frac{\log_2 \log C_n}{2\sqrt{n}}
$$
converges in distribution to a theta random variable as defined by~\eqref{eq:theta}.

(ii) Let $\tau'_n$ be a uniformly random unlabeled unordered binary tree with $n$ leaves, with CP rank $O_n=f(\tau'_n)$. Then there exists a constant $\kappa = \lambda^{-1} \approx 3.13699$
such that
$$
\mathbb{E} \{ \log_2 \log O_n \} \sim \kappa \sqrt{n},
$$
and 
$$
\frac{\log_2 \log O_n}{\kappa\sqrt{n/\pi}}
$$
converges in distribution to a theta random variable as defined by~\eqref{eq:theta}.
\end{theorem}

\begin{proof}
(i) The statement on $\tau_n$ is a consequence of a result of Flajolet and Odlyzko \cite[Theorem B]{FlajoletOdlyzko1982} 
about the height $H_n$ of $\tau_n$: 
$\mathbb{E} \{ H_n \} /\sqrt{n} \to 2 \sqrt{\pi}$ as 
$n \to \infty$,
and ${H_n }/{(2\sqrt{n})}$
tends in distribution to a theta random variable. By Corollary~\ref{cor:CP_vs_height}, the difference $\log_2 \log C_n - H_n$ is $O(1)$, so that
$$\frac{\log_2 \log C_n - H_n}{2\sqrt{n}}$$
is $O(n^{-1/2})$; for any sequence of random trees of increasing size, this quantity goes to $0$ (almost sure convergence, and hence, convergence in probability). The statement on the expected value now follows from linearity of expectation, and the statement on convergence in distribution follows from Slutsky's theorem applied to the convergence in distribution of  ${H_n }/{(2\sqrt{n})}$ and the convergence in probability to 0 of $(\log_2 \log C_n - H_n)/(2\sqrt{n})$.

(ii) The statement on $\tau'_n$ follows in the same fashion from the results of Broutin and Flajolet (\cite[Theorem 1 and Theorem 5]{BroutinFlajolet2008}, \cite[Theorem 1 and Theorem 3]{BroutinFlajolet2012}) on the height of unlabeled unordered binary trees. These state that the height $H'_n$ of a random unlabeled unordered binary tree with $n$ leaves satisfies $\mathbb{E} \{ H'_n \} /\sqrt{n} \to \kappa$, and that ${H'_n }/{(\kappa \sqrt{n/\pi})}$ tends in distribution to a theta random variable. We remark here that our notation is slightly different from that of Broutin and Flajolet: our value of $\lambda$ is $\lambda/(2 \sqrt{\pi})$ in the notation of \cite{BroutinFlajolet2012}, and our distribution function $F(x)$ in~\eqref{eq:theta} is $1-\Theta(2x)$ in their notation.
\end{proof}


\bigskip
\section{Uniformly random leaf-labeled binary trees}
\label{sec:uniform}

A leaf-labeled binary tree with $n$ leaves is a binary tree in which the leaves are bijectively labeled from $1$ to $n$, and in which each internal node has two children. The children are unordered. Such trees are also called \emph{labeled topologies} or \emph{cladograms}. 

We consider a uniformly random cladogram $\tau_n$. The number of such trees is
\begin{equation}
\label{eq:labeled}
(2n-3)\cdot(2n-5)\cdots 3\cdot 1 = \frac{1}{2^{n-1}} \, \frac{(2n-2)!} {(n-1)!},
\end{equation}
all of which are equally likely under this model of randomness. The CP rank of this a random cladogram is denoted by $L_n$ ($L$ for \underline{l}abeled). 

A model of uniformly random cladograms is a special case of more general models on the cladograms, such as Ford's alpha-splitting model~\cite{Ford2005, Ford2006} and Aldous's beta-splitting model~\cite{Aldous1996, Aldous2001}. In particular, Aldous~\cite[Proposition 4, $\beta=-\frac{3}{2}$ case]{Aldous1996} showed that the expected height of a random cladogram satisfies
$$
\mathbb{E} \left\{ H_n \right\} \sim 2 \sqrt{\pi n}.
$$

It is worth pointing out that this result (including the constant $2\sqrt{\pi}$) is the same as for uniformly random unlabeled ordered binary trees 
(compare to Theorem~\ref{uniformtree}i). This is no coincidence: for every unlabeled ordered binary tree on $n$ leaves, there are $n!$ possibilities to label the leaves and turn it into a leaf-labeled ordered binary tree. Likewise, precisely $2^{n-1}$ possibilities turn a labeled unordered binary tree on $n$ leaves into a labeled ordered binary tree (by switching the order of the children at the internal nodes). For this reason, the distribution of the height and any other parameters that do not depend on labels or order is the same for three uniform models: unlabeled ordered, labeled unordered, and labeled ordered binary trees~\cite[Section 3.1]{Disanto2022}. In particular, the following result is equivalent to part (i) of Theorem~\ref{uniformtree}.

\begin{theorem}\label{cladogram}
Let $\tau_n$ be a uniformly random leaf-labeled binary tree with $n$ leaves, with CP rank $L_n=f(\tau_n)$. Then
$$
\mathbb{E} \{ \log_2 \log L_n \} \sim 2 \sqrt{\pi n},
$$
and 
$$
\frac{\log_2 \log L_n}{2\sqrt{n}}
$$
converges to a theta distribution.
\end{theorem}


Aldous's beta-splitting model for random binary trees has a shape parameter $\beta \in [-2, \infty]$. The special case $\beta = -\frac32$ corresponds to the uniform model of Theorem~\ref{cladogram}. Generally, in \cite[Proposition 4]{Aldous1996}, Aldous proved the following results on the height $H_n$:
\begin{itemize}
\item For $\beta > -1$, the ratio $H_n / \log n$ tends in probability and in expectation to a constant $g(\beta)$. There is no explicit expression for this constant in general, but numerical values can be determined from an implicit equation given in \cite[Proposition 4]{Aldous1996}.

To mention some examples, $g(\infty) = 1/\log 2 \approx 1.44270$, and we obtain $g(1) \approx 3.19258$,
$g(0) \approx 4.31107$, and $g(-\frac{1}{2}) \approx 6.38090$ from the implicit equation (note that \cite{Aldous1996} only gives two digits each). 

The case $\beta=0$ corresponds to the Yule model (see Section 
\ref{sec:Yule} below for more information). For $\beta = \infty$, all internal nodes split their subtrees (almost) precisely in half: the difference of the subtree sizes is at most $1$.
\item For $\beta = -1$, $\mathbb{E} \{ H_n \} \ge
\big(6/\pi^2 + o(1)\big)(\log n)^2$ (Aldous's proposition did not report a result for $\mathbb{E} \{ H_n \}$ with $\beta=-1$, but this inequality follows quickly from Aldous's results reported in the proposition for related quantities). 
\item For $\beta \in (-2,-1)$, $n^{1+\beta} \mathbb{E} \{H_n \} \to g(\beta)$,
and $n^{1+\beta}  H_n$ has a non-degenerate limit distribution. 
\end{itemize}

These results on tree height for cladograms under the beta-splitting model directly impact the Colijn--Plazzotta rank. For example, for $\beta \in (-2,-1)$, we have
for Aldous's beta-splitting tree $\tau_n$ with $n$ leaves
$$
\mathbb{E} \{ \log_2 \log f(\tau_n) \}\sim \frac{g(\beta)}{n^{1+\beta}} .
$$

\bigskip
\section{Yule--Harding trees, random binary search trees}
\label{sec:Yule}

Among the probability distributions that could be placed on the leaf-labeled binary trees with $n$ leaves, perhaps the most frequently considered, along with the uniform distribution of Section \ref{sec:uniform}, is the $\beta=0$ case of the beta-splitting model. This model corresponds to the random binary search trees, which are identical to Yule or Yule--Harding trees in phylogenetics~\cite{Fuchs2025}, except for the convention that random binary search trees are typically indexed by the number of internal nodes and Yule--Harding trees are indexed by the number of leaves. We index trees by the number of leaves, considering random binary search trees in which all internal nodes have two children so that the total number of internal nodes is $n-1$ when the number of leaves is $n$.

To be precise, we start with a standard random binary search tree on $n-1$ (internal) nodes and attach a layer of $n$ external nodes, i.e., we give a second child to all (internal) nodes having one child, and give two children to all leaves. The random CP rank of a tree under this model is denoted by $S_n$ ($S$ for \underline{s}earch tree).

For these trees, the height $H_n$ satisfies~\cite[Theorem 5.1]{Devroye1986}
$$
\frac{H_n}{\log n} \xrightarrow{p} \alpha,
$$
where $\alpha \approx 4.31107$ is the unique solution in $(2, \infty)$ of the equation
$$
\alpha \log (2e/\alpha) = 1.
$$
Setting
$$
\beta = \frac{3 \alpha}{2 \alpha -2 } \approx 1.95303,
$$
Reed~\cite[Lemma~8 and Lemma~10]{reed2003} and Drmota~\cite[Theorem~2.5]{drmota2003} showed that 
$H_n - \alpha \log n + \beta \log \log n$ is tight, i.e., 
\begin{equation}\label{eq:tight}
\limsup_{x \uparrow \infty} \, \big[\sup_n \mathbb{P} \left\{ | H_n - \alpha \log n + \beta \log \log n | \ge x \right\}\big] = 0.
\end{equation}
Indeed, Reed~\cite[Lemma~8]{reed2003} states that there exists a universal constant $C_2$ such that for $i >0$ with
$i = o(\log n)$,
$$
\mathbb{P} \{ H_n \geq \alpha \log n - \beta \log \log n + i \} < C_2 \cdot 2^{-i/2}.
$$
In the other direction, Reed~\cite[Lemma~10]{reed2003} states that there exist universal constants $C_1,C_3$ and $d < 1$ such that for $i >0$ with
$i = o(\sqrt{\log n})$,
$$
\mathbb{P} \{ H_n \leq \alpha \log n - \beta \log \log n - C_3 - i \} < C_1 \cdot d^{i}.
$$
Combining the two directions, we find that, for $x \geq C_3$,
$$\mathbb{P} \left\{ | H_n - \alpha \log n + \beta \log \log n | \ge x \right\} \leq C_2 \cdot 2^{-x/2} + C_1 \cdot d^{x-C_3}$$
holds for all $n$. The bound on the right side is independent of $n$ and goes to $0$ as $x \uparrow \infty$, so~\eqref{eq:tight} follows.


\begin{theorem}\label{rbst}
Let $\tau_n$ be a random leaf-labeled binary tree with $n$ leaves following the Yule--Harding distribution, with CP rank $S_n=f(\tau_n)$. Then
\[\mathbb{E} \{\log_2 \log S_n \} \sim \alpha \log n,\]
and 
$$
\frac{\log_2 \log S_n}{\log n} \overset{p}{\to} \alpha.
$$
\end{theorem}

\begin{proof}
The proof is similar to Theorem~\ref{uniformtree}.
By Corollary~\ref{cor:CP_vs_height}, the difference between $\log_2 \log S_n$ and the height $H_n$ is bounded, so 
$$
\frac{\log_2 \log S_n - H_n}{\log n}
$$
goes to $0$ (almost surely, thus also in probability). 
So the result follows immediately via Slutsky's theorem from the fact that $H_n / \log n \overset{p}{\to} \alpha$.
\end{proof}

\begin{theorem}\label{rbst2}
Let $\tau_n$ be a random leaf-labeled binary tree with $n$ leaves following the Yule--Harding distribution, with CP rank $S_n=f(\tau_n)$. Then
$$
\frac {(\log n)^{\beta \log 2}  \log S_n} {n^{\alpha \log 2}}
$$
is a tight sequence of random variables.
\end{theorem}

\begin{proof}
By Corollary~\ref{cor:CP_vs_height}, there exists an absolute positive constant $K$ such that $K\cdot 2^{H_n} \geq \log S_n$. Thus,
$$\frac{(\log n)^{\beta \log 2}  \log S_n}{n^{\alpha \log 2}} \geq x$$
implies 
$$2^{H_n} \geq \frac{x n^{\alpha \log 2}}{K (\log n)^{\beta \log 2}},$$
or
$$H_n - \alpha \log n + \beta \log \log n \geq \frac{\log(x/K)}{\log 2}.$$
This means that
\begin{align*}
\mathbb{P} \Big\{ \Big| \frac{(\log n)^{\beta \log 2}  \log S_n} {n^{\alpha \log 2}} \Big| \geq x \Big\}
&= \mathbb{P} \Big\{ \frac{(\log n)^{\beta \log 2}  \log S_n} {n^{\alpha \log 2}} \geq x \Big\} \\
&\leq \mathbb{P} \Big\{ H_n - \alpha \log n + \beta \log \log n \geq \frac{\log(x/K)}{\log 2} \Big\} \\
&\leq \mathbb{P} \Big\{ |H_n - \alpha \log n + \beta \log \log n| \geq \frac{\log(x/K)}{\log 2} \Big\}.
\end{align*}
By~\eqref{eq:tight}, this expression goes to $0$ if we take $\sup_n$ and then $\limsup_{x \uparrow \infty}$, showing that the sequence is indeed tight.
\end{proof}

\bigskip
\section
{Mean and variance of the Colijn--Plazzotta rank}

Sections \ref{sec:unlabeled}--\ref{sec:Yule} focus on properties of the distribution of $\log \log f(\tau_n)$ under various models of randomness; in this section, we focus on the distribution of the random CP rank $f(\tau_n)$ itself. In particular, we study the first-order asymptotics of the mean and variance of the Colijn--Plazzotta rank under the models of randomness from Sections \ref{sec:unlabeled}--\ref{sec:Yule}, investigating $C_n$, $O_n$, $L_n$, and $S_n$. As pointed out in Section~\ref{sec:uniform}, the models of uniformly random unlabeled ordered binary trees (Catalan trees) and uniformly random labeled unordered binary trees are equivalent for our purposes, so that the distributions of $C_n$ and $L_n$ are the same.

We give a general theorem on the mean and variance of the Colijn--Plazzotta rank applicable to all random tree models specifying a certain condition. We then obtain first-order asymptotics for the means and variances of $C_n$, $O_n$, $L_n$ and $S_n$ as simple corollaries. The desired means and variances are determined mainly by the extreme cases for Colijn--Plazzotta ranks. 


\begin{lemma}\label{maximum}
(i) Among all unlabeled binary rooted trees with $n$ leaves, $n \geq 1$, the Colijn--Plazzotta rank is maximized by the caterpillar.
(ii) Among all unlabeled binary rooted trees with $n$ leaves and height $n-2$ or less, $n \geq 4$, the Colijn--Plazzotta rank is maximized by the pseudocaterpillar.
\end{lemma}
\begin{proof}
    (i) This result was proven in Corollary 20 of \cite{Rosenberg2021}. (ii) This result follows by induction and Lemma \ref{rank-height}. For $n=4$, the pseudocaterpillar is the only tree with height at most $n-2=2$. Suppose for induction that for all $k$, $4 \leq k \leq n-1$, the pseudocaterpillar has the maximal Colijn--Plazzotta rank among trees with $k$ leaves and height $k-2$. 
    
    Among trees $t$ with $n$ leaves and height at most $n-2$, by definition of the Colijn--Plazzotta rank, the rank $f(t)$ is maximized by choosing its left subtree $\ell(t)$ to have $f\big(\ell(t)\big)$ as large as possible. The left subtree $\ell(t)$ has at most $n-1$ leaves and height at most $n-3$, so that the inductive hypothesis applies: $\ell(t)$ is the pseudocaterpillar with $n-1$ leaves, the right subtree $r(t)$ is  a single leaf, and $t$ is the pseudocaterpillar with $n$ leaves.
\end{proof}
For the following theorem, we recall Rosenberg's~\cite{Rosenberg2021} sequence for the maximal Colijn--Plazzotta rank $c_h$ among trees with height $h \geq 0$ and $h+1$ leaves: $c_0 = 1$, and 
\begin{equation}
\label{eq:recursion-c}
c_{h+1} = \binom{c_h}{2} + 2, \, h \ge 0.
\end{equation}
Equivalently, $c_h$ is the Colijn--Plazzotta rank of a caterpillar
of height $h$. Recall that $c_2 = 3$, $c_3 = 5$, $c_4 = 12$, and $c_5 = 68$.

We also let $d_h$ be the corresponding rank of a pseudocaterpillar of height $h$. Then $d_2=4$, and 
\begin{equation}
\label{eq:recursion-d}
d_{h+1} = \binom{d_h}{2} + 2, \, h \ge 2.
\end{equation}
The sequences $c_h$ and $d_h$ obey identical recursions, only with different starting points. Sequence $d_h$ begins with $d_2=4$, $d_3=8$, $d_4=30$, and $d_5=437$.

\begin{theorem}\label{mean+variance}
For a given probability model for random binary rooted trees $T_n$ with $n$ leaves, let
$$
\pi_n = \mathbb {P} \left\{ T_n ~\textrm{is a caterpillar} \right\},
$$
and let $P_n$ be the Colijn--Plazzotta rank of $T_n$.
If $\pi_n = o(1)$ and
\begin{align}\label{sufficient}
\log (1/\pi_n) = o(2^n),
\end{align}
then
\begin{align*}
\mathbb {E} \left\{ P_n \right\} & \sim \pi_n c_{n-1}, \\
\V \left\{ P_n \right\}  \sim 
\mathbb {E} \left\{ P_n^2 \right\} & \sim 
\pi_n c_{n-1}^2.
\end{align*}
\end{theorem}

\begin{proof}
Trivially, extracting terms in the sum $\mathbb {E} \left\{ P_n \right\}$ corresponding to the events that $T_n$ is a caterpillar of height $n-1$ or a pseudocaterpillar of height $n-2$,
\begin{align}
\label{eq:EPn}
\pi_n c_{n-1} \le \mathbb {E} \left\{ P_n \right\} & \le \pi_n c_{n-1} + d_{n-2}, \\ 
\label{eq:EPn2}
\pi_n c_{n-1}^2 \le \mathbb {E} \left\{ P_n^2 \right\} & \le \pi_n c_{n-1}^2 + d_{n-2}^2.
\end{align}
By taking the ratio of \eqref{eq:EPn} with $\pi_n c_{n-1}$, to verify $\mathbb {E} \left\{ P_n \right\} \sim \pi_n c_{n-1}$, it suffices to show
\begin{align}\label{condition2}
\lim_{n \to \infty} \frac{ d_{n-2}} {\pi_n c_{n-1}} = 0.
\end{align}
Similarly, because $d_{n-2} < c_{n-1}$ so that $(d_{n-2} / c_{n-1})^2 < d_{n-2} / c_{n-1}$, by taking the ratio of  \eqref{eq:EPn2} and $\pi_n c_{n-1}^2$, verifying condition \eqref{condition2} suffices for verifying
$\V \left\{ P_n \right\}  \sim 
\mathbb {E} \left\{ P_n^2 \right\} \sim 
\pi_n c_{n-1}^2$; we see first that $\mathbb{E} \left\{ P_n^2 \right\} \sim \pi_n c_{n-1}^2$, and then $\V \left\{ P_n \right\} =\mathbb {E} \left\{ P_n^2 \right\} - \mathbb {E} \left\{ P_n \right\}^2 \sim 
\mathbb {E} \left\{ P_n^2 \right\}$ follows by recalling that $\pi_n = o(1)$.

We will show that (\ref{sufficient}) implies (\ref{condition2}). We first prove by induction that $d_{h-1} < 0.9^{2^{h-3}} c_h$ for all $h \geq 3$. This statement is readily verified for $h = 3$ and $h=4$. Now assume that the inequality holds for some positive integer $h \geq 4$, and write $Q_h = 0.9^{-(2^{h-3})} > 1$, so that $c_h > Q_h d_{h-1}$. It follows from the recursions \eqref{eq:recursion-c} and \eqref{eq:recursion-d} that
\begin{align*}
\frac{d_h}{c_{h+1}} = \frac{d_{h-1}^2 - d_{h-1} + 4}{c_h^2 - c_h + 4} &< \frac{d_{h-1}^2 - d_{h-1} + 4}{Q_h^2 d_{h-1}^2 - Q_h d_{h-1} + 4} \\
&= \frac{1}{Q_h^2} - \frac{(Q_h-1)(Q_h d_{h-1}-4Q_h-4)}{Q_h^2(Q_h^2 d_{h-1}^2 - Q_h d_{h-1} + 4)}.    
\end{align*}
The final fraction is positive since $Q_h > 1$ and $d_{h-1} \geq d_3 \geq 8$. Thus,
$$\frac{d_h}{c_{h+1}} < \frac{1}{Q_h^2} = 0.9^{2^{h-2}},$$
completing the induction.

It follows (for $n \geq 4$) that 
$$\log \frac{d_{n-2}}{\pi_n c_{n-1}} \leq \log \Big( \frac{1}{\pi_n} 0.9^{2^{n-4}}\Big) = 2^{n-4} \log 0.9 - \log \pi_n = 2^{n-4} \log 0.9 + o(2^n)$$
by the assumption~\eqref{sufficient}. Because this last expression goes to $-\infty$ as $n$ increases without bound, we have verified~\eqref{condition2}. This completes the proof.
\end{proof}

The theorem finds that the expectation and variance of the CP rank are determined asymptotically by the CP rank of the caterpillar; the asymptotic mean is simply the product of the CP rank of the caterpillar and the probability that a tree is a caterpillar. In all four types of random trees that we consider, we verify that $\pi_n$ satisfies \eqref{sufficient}, so that the theorem applies. This verification amounts to demonstrating that caterpillars are sufficiently probable as $n$ grows large; if $\pi_n$ were to decrease too quickly, then the condition would not be satisfied. 

The number of caterpillar cladograms is $n!/2$, so that 
for a random cladogram (and equivalently, for a random Catalan tree), \eqref{eq:labeled} gives
\begin{equation}
\label{eq:piCatalan}
\pi_n = \frac{n!}{2} \frac{1}{(2n-3)!!} = \frac{2^{n-2}}{\frac{1}n \binom{2n-2}{n-1}}
\sim \frac{2^{n-2}} {\pi^{-1/2}n^{-3/2}4^{n-1}}
\sim \frac{n^{3/2} \sqrt{\pi}}{2^n}.
\end{equation}
For a random Otter tree on $n$ leaves, we have no simple explicit expression for $\pi_n$. However, we have the asymptotic probability from (\ref{ottercount}) that a random Otter tree is the unique caterpillar:
\begin{equation}
\label{eq:piOtter}
\pi_n
\sim \frac{n^{3/2} \rho^{n}}{\lambda}.
\end{equation}
Finally, for a random binary search tree~\cite[p.~92]{Slowinski1990},
\begin{equation}
\label{eq:pirbst}
\pi_n = \frac{n!}{2} \frac{1}{\frac{n! \, (n-1)!}{2^{n-1}}} = \frac{2^{n-2}}{(n-1)!} \sim \bigg( \frac{2e}{n} \bigg)^n \frac{\sqrt{n}}{4\sqrt{2\pi}}.
\end{equation}

Verifying in \eqref{eq:piCatalan}, \eqref{eq:piOtter}, and \eqref{eq:pirbst} that condition \eqref{sufficient} is satisfied, we have shown the following theorem.
\begin{theorem}\label{mean+variance2}
With $\pi_n$ as in \eqref{eq:piCatalan}, \eqref{eq:piOtter}, and  \eqref{eq:pirbst}, and with $P_n$ corresponding to either
 $C_n$ (the random Catalan tree), $O_n$ (the random Otter tree), 
$L_n$ (the random cladogram), or $S_n$ (the random binary search tree), we have
\begin{align*}
\mathbb {E} \left\{ P_n \right\} & \sim \pi_n c_{n-1}, \\
\V \left\{ P_n \right\} \sim 
\mathbb {E} \left\{ P_n^2 \right\} & \sim 
\pi_n c_{n-1}^2.
\end{align*}
\end{theorem}

\bigskip
\section{Numerical computations}

We informally examine the extent to which the asymptotic approximations for $\mathbb{E}\{\log_2 \log f(\tau_n)\}$, $\mathbb{E}\{ f(\tau_n) \}$, and $\mathbb{V}\{ f(\tau_n) \}$ agree with the exact values for small $n$. First, Tables \ref{table:1} and \ref{table:2} show the CP rank and the probabilities of all unlabeled unordered binary trees for $n=1$ to 8 under each of three models: uniformly random unlabeled unordered trees, uniformly random leaf-labeled trees, and Yule--Harding leaf-labeled trees. The much larger CP rank for the caterpillar compared to the pseudocaterpillar (and all other trees) is already visible for $n=8$.

Figure \ref{fig:1} plots the values of $\mathbb{E}\{\log_2 \log f(\tau_n)\}$, the mean height $H_n$, and the asymptotic approximation for $\mathbb{E}\{\log_2 \log f(\tau_n)\}$ under the three models. For each of the three models, we can observe similar shapes in plots for its three quantities. The values are greatest for the uniformly random leaf-labeled trees, with asymptotic approximation $2 \sqrt {\pi n} \approx 3.54491 \sqrt{n}$, followed by the uniformly random unlabeled unordered trees, with asymptotic approximation $3.13699 \sqrt{n}$, and finally, the Yule--Harding leaf-labeled trees, with asymptotic approximation $4.31107 \log n$. 

Figures \ref{fig:2} and \ref{fig:3} plot the exact mean and variance of $f(\tau_n)$ under the three models alongside the asymptotic approximation based on the contribution of the caterpillar tree, taking the $\log_2 \log$ of these quantities to produce a comparable scale to Figure \ref{fig:1}. In the figure, we observe that even for quite small $n$, the exact mean and variance are closely approximated by the asymptotic $\pi_n c_{n-1}$. The mean and variance are greatest for the uniformly random leaf-labeled trees, for which $\pi_n \sim \sqrt{\pi} ( n^{3/2} )(0.5^n)$ \eqref{eq:piCatalan}, followed by the uniformly random unlabeled unordered trees, with asymptotic approximation $\pi_n \sim \lambda^{-1} n^{3/2} \rho^{-n} \approx 3.13699 (n^{3/2})(0.40270^n)$ \eqref{eq:piOtter}. For the Yule--Harding model, caterpillars are least probable \eqref{eq:pirbst}.

\bigskip
\section{Discussion}

We have analyzed the Colijn--Plazzotta rank of rooted binary trees, showing that the rank of a tree is largely determined by its height. Indeed, the ranking proceeds through all trees of a given height $h$ before moving on to trees of height $h+1$. We have also obtained asymptotic properties of the trees under three different models for selecting random trees, finding in particular the asymptotics of $\E \{\log_2 \log f(\tau_n) \}$ for random trees $\tau_n$. The asymptotic mean and variance of the CP rank across trees with $n$ leaves depend only on the probability and CP rank of the $n$-leaf caterpillar, as the product of the probability and the CP rank of the caterpillar grows faster than the next-highest rank. A summary of mathematical results appears in Table~\ref{table:3}. 

Numerical investigations clarify a pattern observable in the mathematical results, namely that the ``uniform'' model---uniformly random leaf-labeled trees---has CP ranks greater than the Yule--Harding model on leaf-labeled trees (Figures \ref{fig:1}--\ref{fig:3}). This observation can be viewed as a consequence of the greater probability of the caterpillar shape in the uniform \eqref{eq:piCatalan} than in the Yule--Harding model \eqref{eq:pirbst}. 

It has been suggested that CP rank can serve as a measure of tree balance and imbalance in empirical studies~\cite{FischerEtAl2023, Rosenberg2021}. We have found that as $n$ grows, the CP rank of the caterpillar grows so fast that for both the uniform and Yule--Harding models on leaf-labeled trees, the mean CP rank across trees with $n$ leaves is asymptotically determined by the contribution of the caterpillar. Hence, as a balance statistic beyond the smallest tree sizes, the use of CP rank $f(\tau)$ would amount primarily to distinguishing caterpillars from non-caterpillars. A potentially more suitable statistic is $\log_2 \log f(\tau)$, which places the CP ranks of different trees on a similar scale.

The results have been obtained by connecting studies of CP rank as a quantity of mathematical phylogenetics to the extensive literature on tree height in studies grounded in theoretical computer science. As has been demonstrated here, such applications of theoretical computer science results on tree properties have the potential to provide solutions to unsolved problems in mathematical phylogenetics. 

Although we have obtained the asymptotics of the mean and variance of the CP rank under the uniform and Yule--Harding models---the two models for which the mean and variance were noted by \cite{FischerEtAl2023} as open problems---we have not commented on the \emph{exact} mean and variance. For practical applications of CP rank, an understanding of the asymptotics likely suffices, but we note that the precise determination of the mean and variance of the CP rank remains an open problem.

\medskip
\noindent \textsc{Acknowledgments.}
This project developed from conversations at the Analysis of Algorithms meeting in Bath, United Kingdom (AofA2024), and we are grateful to the conference organizers. 

\medskip
\noindent \textsc{Funding.} We acknowledge the Natural Sciences and Engineering Research Council of Canada (LD), National Institutes of Health grant R01 HG005855 (NAR), and Swedish Research Council/Vetenskapsr\aa det grant 2022-04030 (SW).

\medskip
\noindent \textsc{Data availability statement.} The study has no associated data.

\bigskip
\bibliographystyle{abbrv}
\bibliography{ColijnPlazzotta}

\clearpage
\begin{table}[tb]
\begin{tabular}{|c|P{2cm}|c|c|P{3.5cm}|P{2.2cm}|P{2.7cm}|}
 \hline 
     & & & & \multicolumn{3}{|c|}{Model} \\ \cline{5-7}
 $n$ & \centering $t_n$ & $f(t_n)$ & Height & Unlabeled uniform unordered & Leaf-labeled uniform & Leaf-labeled Yule--Harding \\ \hline 
\footnotesize
1 & \centering \includegraphics[scale=0.16]{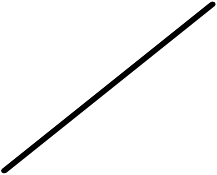}    &    1 & 0 &    1 & 1    & 1    \\
2 & \centering \includegraphics[scale=0.16]{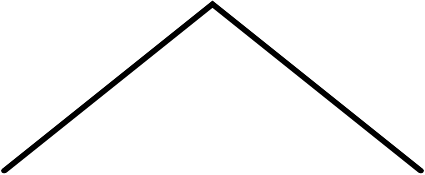}    &    2 & 1 &    1 & 1    & 1    \\
3 & \centering \includegraphics[scale=0.16]{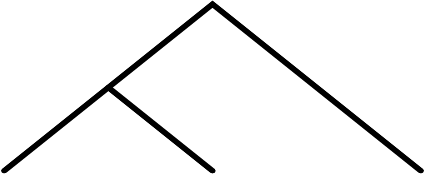}    &    3 & 2 &    1 & 1    & 1    \\
4 & \centering \includegraphics[scale=0.16]{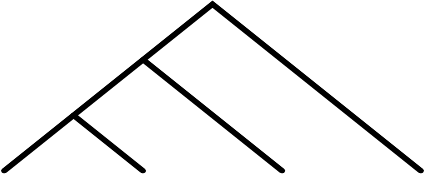}  &    5 & 3 &  1/2 & 4/5  & 2/3  \\
4 & \centering \includegraphics[scale=0.16]{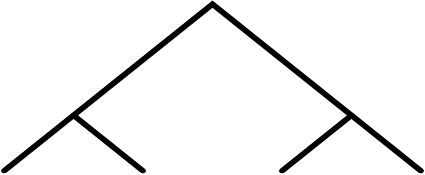}  &    4 & 2 &  1/2 & 1/5  & 1/3  \\
5 & \centering \includegraphics[scale=0.16]{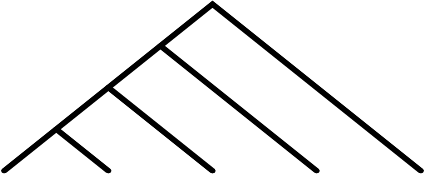}  &   12 & 4 &  1/3 & 4/7  & 1/3  \\
5 & \centering \includegraphics[scale=0.16]{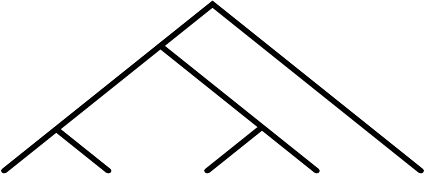}  &    8 & 3 &  1/3 & 1/7  & 1/6  \\
5 & \centering \includegraphics[scale=0.16]{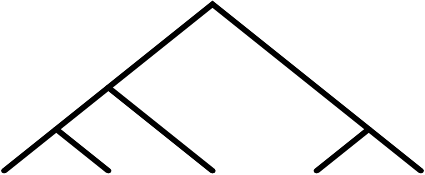}  &    6 & 3 &  1/3 & 2/7  & 1/2  \\
6 & \centering \includegraphics[scale=0.16]{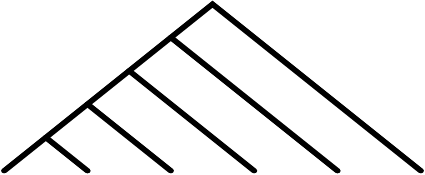}  &   68 & 5 &  1/6 & 8/21 & 2/15 \\
6 & \centering \includegraphics[scale=0.16]{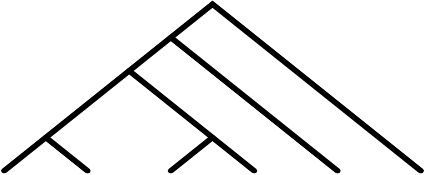}  &   30 & 4 &  1/6 & 2/21 & 1/15 \\
6 & \centering \includegraphics[scale=0.16]{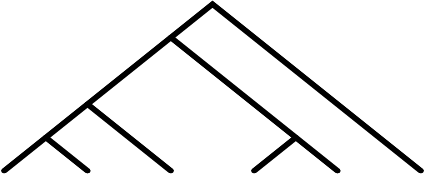}  &   17 & 4 &  1/6 & 4/21 & 1/5  \\
6 & \centering \includegraphics[scale=0.16]{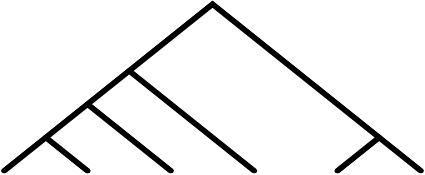}  &   13 & 4 &  1/6 & 4/21 & 4/15 \\
6 & \centering \includegraphics[scale=0.16]{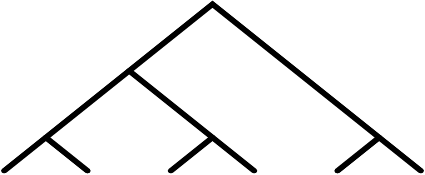}  &    9 & 3 &  1/6 & 1/21 & 2/15 \\
6 & \centering \includegraphics[scale=0.16]{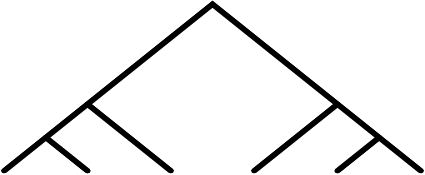}  &    7 & 3 &  1/6 & 2/21 & 1/5  \\
7 & \centering \includegraphics[scale=0.16]{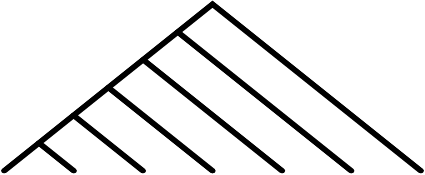}  & 2280 & 6 & 1/11 & 8/33 & 2/45 \\
7 & \centering \includegraphics[scale=0.16]{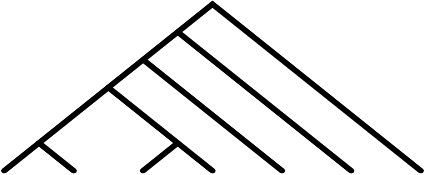}  &  437 & 5 & 1/11 & 2/33 & 1/45 \\
7 & \centering \includegraphics[scale=0.16]{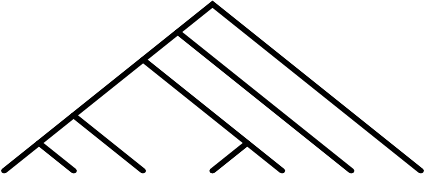}  &  138 & 5 & 1/11 & 4/33 & 1/15 \\
7 & \centering \includegraphics[scale=0.16]{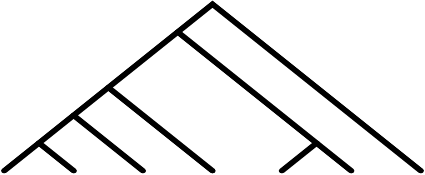}  &   80 & 5 & 1/11 & 4/33 & 4/45 \\
7 & \centering \includegraphics[scale=0.16]{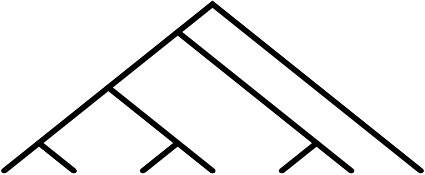}  &   38 & 4 & 1/11 & 1/33 & 2/45 \\
7 & \centering \includegraphics[scale=0.16]{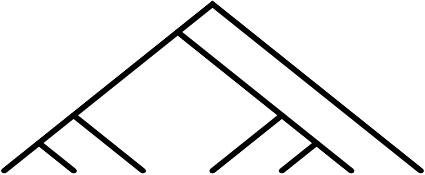}  &   23 & 4 & 1/11 & 2/33 & 1/15 \\
7 & \centering \includegraphics[scale=0.16]{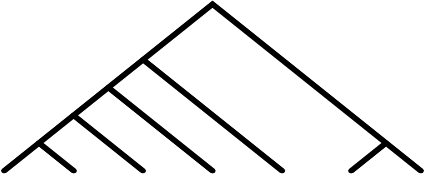}  &   69 & 5 & 1/11 & 4/33 & 1/9  \\
7 & \centering \includegraphics[scale=0.16]{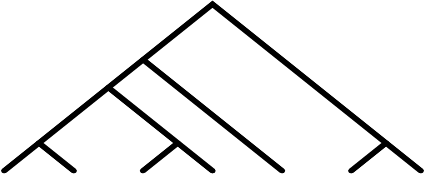}  &   31 & 4 & 1/11 & 1/33 & 1/18 \\
7 & \centering \includegraphics[scale=0.16]{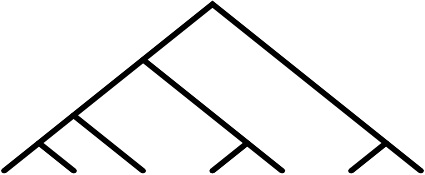}  &   18 & 4 & 1/11 & 2/33 & 1/6  \\
7 & \centering \includegraphics[scale=0.16]{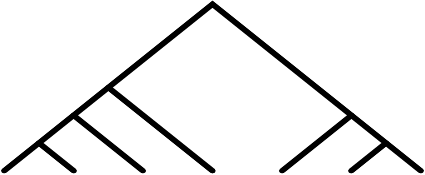} &   14 & 4 & 1/11 & 4/33 & 2/9  \\
7 & \centering \includegraphics[scale=0.16]{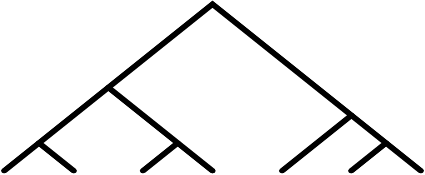} &   10 & 3 & 1/11 & 1/33 & 1/9  \\
\hline
\end{tabular}
\caption{CP rank $f(t_n)$ and probability under three models for all unlabeled unordered binary trees $t_n$ with $n$ leaves, $1 \leq n \leq 7$. For unlabeled uniform unordered trees, the probability is the reciprocal of the number of such trees, the Wedderburn--Etherington number (OEIS A001190 \cite{OEIS}). For leaf-labeled uniform trees, it is the ratio of $n! / 2^{s(t_n)}$ (the number of ways of labeling shape $t_n$, where the number of symmetric nodes $s(t_n)$ is the number of internal nodes whose two descendant subtrees have the same unlabeled shape) and $(2n-3)!!$, the number of leaf-labeled trees with $n$ leaves \eqref{eq:labeled}. For leaf-labeled Yule--Harding trees, it is the ratio of $[n! / 2^{s(t_n)}][(n-1)!/\prod_{r=2}^n (r-1)^{d_r(t_n)}]$ and $n!(n-1)!/2^{n-1}$, where $d_r(t_n)$ is the number of internal nodes of $t_n$ with $r$ descendant leaves, $(n-1)!/\prod_{r=2}^n (r-1)^{d_r(t_n)}$ gives the number of \emph{labeled histories} of a leaf-labeled tree (the number of sequences in which the tree can be produced by a sequence of bifurcations), and  $n!(n-1)!/2^{n-1}$ is the total number of labeled histories for $n$ labeled leaves.}
\label{table:1}
\end{table}
\clearpage
\begin{table}[tb]
\begin{tabular}{|c|P{2cm}|c|c|P{3.5cm}|P{2.2cm}|P{2.7cm}|}
 \hline 
     & & & & \multicolumn{3}{|c|}{Model} \\ \cline{5-7}
 $n$ & \centering $t_n$ & $f(t_n)$ & Height & Unlabeled uniform unordered & Leaf-labeled uniform & Leaf-labeled Yule--Harding \\ \hline 
\footnotesize
8 & \centering \includegraphics[scale=0.2]{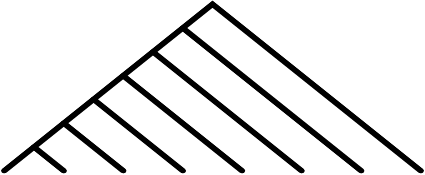}  & 2598062 & 7 & 1/23 & 64/429 & 4/315 \\
8 &\centering \includegraphics[scale=0.2]{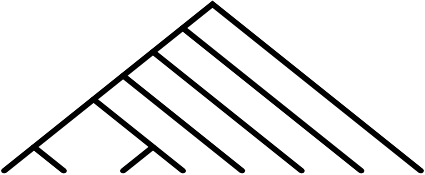}   & 95268   & 6 & 1/23 & 16/429 & 2/315 \\
8 &\centering \includegraphics[scale=0.2]{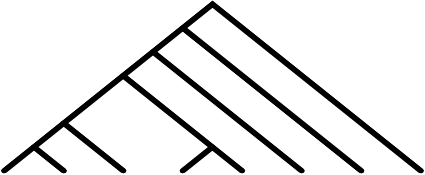}   & 9455    & 6 & 1/23 & 32/429 & 2/105 \\
8 &\centering \includegraphics[scale=0.2]{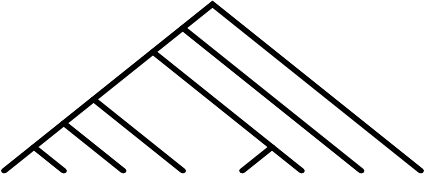}   & 3162    & 6 & 1/23 & 32/429 & 8/315 \\
8 &\centering \includegraphics[scale=0.2]{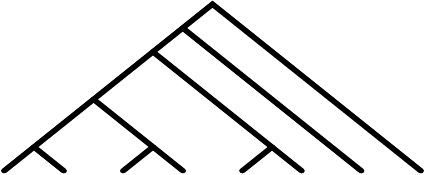}   & 705     & 5 & 1/23 & 8/429  & 4/315 \\
8 &\centering \includegraphics[scale=0.2]{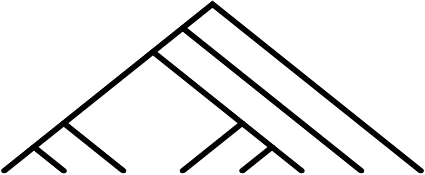}   & 255     & 5 & 1/23 & 16/429 & 2/105 \\
8 &\centering \includegraphics[scale=0.2]{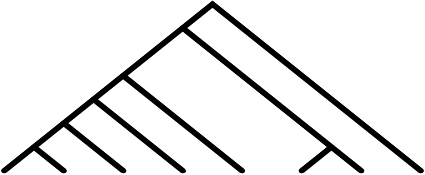}   & 2348    & 6 & 1/23 & 32/429 & 2/63  \\
8 &\centering \includegraphics[scale=0.2]{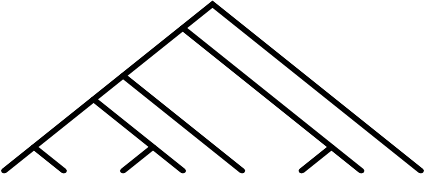}   & 467     & 5 & 1/23 & 8/429  & 1/63  \\
8 &\centering \includegraphics[scale=0.2]{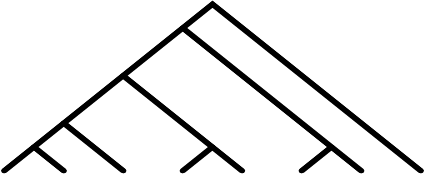}   & 155     & 5 & 1/23 & 16/429 & 1/21  \\
8 &\centering \includegraphics[scale=0.2]{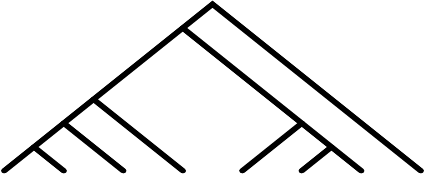}  & 93      & 5 & 1/23 & 32/429 & 4/63  \\
8 &\centering \includegraphics[scale=0.2]{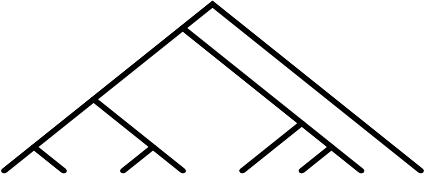}  & 47      & 4 & 1/23 & 8/429  & 2/63  \\
8 &\centering \includegraphics[scale=0.2]{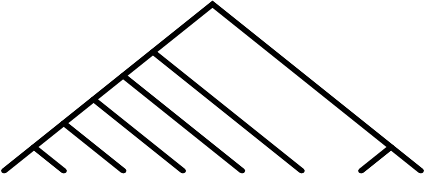}  & 2281    & 6 & 1/23 & 32/429 & 4/105 \\
8 &\centering \includegraphics[scale=0.2]{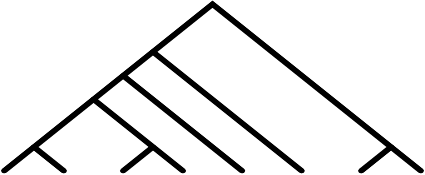}  & 438     & 5 & 1/23 & 8/429  & 2/105 \\
8 &\centering \includegraphics[scale=0.2]{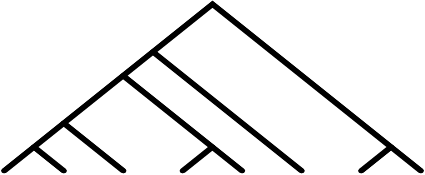}  & 139     & 5 & 1/23 & 16/429 & 2/35  \\
8 &\centering \includegraphics[scale=0.2]{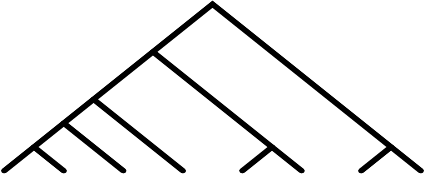}  & 81      & 5 & 1/23 & 16/429 & 8/105 \\
8 &\centering \includegraphics[scale=0.2]{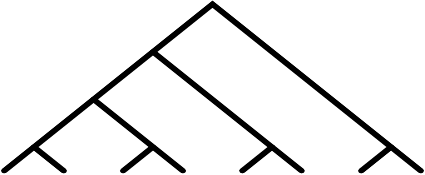}  & 39      & 4 & 1/23 & 4/429  & 4/105 \\
8 &\centering \includegraphics[scale=0.2]{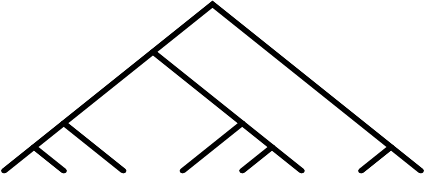}  & 24      & 4 & 1/23 & 8/429  & 2/35  \\
8 &\centering \includegraphics[scale=0.2]{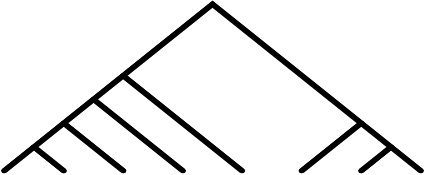}  & 70      & 5 & 1/23 & 32/429 & 2/21  \\
8 & \centering \includegraphics[scale=0.2]{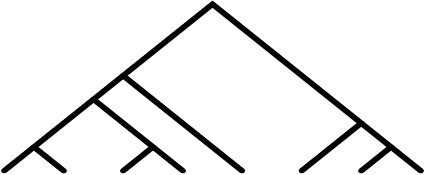} & 32      & 4 & 1/23 & 8/429  & 1/21  \\
8 &\centering \includegraphics[scale=0.2]{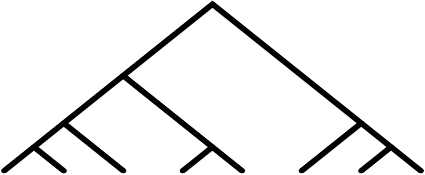}  & 19      & 4 & 1/23 & 16/429 & 1/7   \\
8 &\centering \includegraphics[scale=0.2]{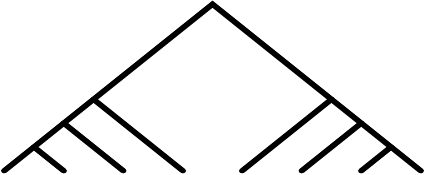}  & 16      & 4 & 1/23 & 16/429 & 4/63  \\
8 &\centering \includegraphics[scale=0.2]{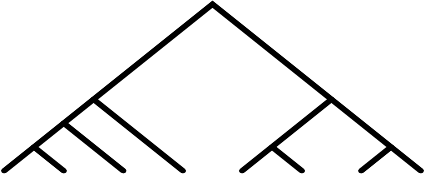}  & 15      & 4 & 1/23 & 8/429  & 4/63  \\
8 &\centering \includegraphics[scale=0.2]{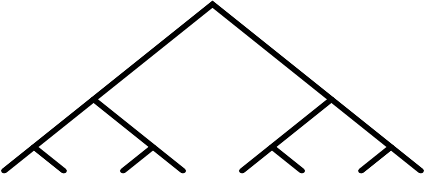}  & 11      & 3 & 1/23 & 1/429  & 1/63  \\ \hline
\end{tabular}
\caption{CP rank $f(t_n)$ and probability under three models for all unlabeled unordered binary trees $t_n$ with $n$ leaves, $n=8$. The table design follows Table \ref{table:1}.}
\label{table:2}
\end{table}
\clearpage
\begin{table}[tbh]
    \centering
    \begin{tabular}{|p{3cm}|p{3.5cm}|p{3.5cm}|p{3.5cm}|}
        \hline
         & \multicolumn{3}{|c|}{Model} \\ \cline{2-4}
\hfil Property \hfil&  \hfil Unlabeled  \hfil & \hfil Leaf-labeled \hfil & \hfil Leaf-labeled \hfil \\ 
&   \hfil uniform unordered \hfil & \hfil uniform \hfil  & \hfil Yule--Harding \hfil \\
\hline
\phantom{M} & & & \\
\hfil$\E\{\log \log f(\tau_n)\}$\hfil & \hfil Theorem \ref{uniformtree}\hfil & \hfil Theorem \ref{cladogram} \hfil& \hfil Theorem \ref{rbst} \hfil\\
\phantom{M} & & & \\
\hfil$\log f(\tau_n)$\hfil  &                \hfil - \hfil                        &                 \hfil    - \hfil                   & \hfil Theorem \ref{rbst2}\hfil \\
\phantom{M} & & & \\
\hfil$\E\{f(\tau_n)\}$\hfil & \hfil Theorem \ref{mean+variance} \hfil & \hfil Theorem \ref{mean+variance} \hfil & \hfil Theorem \ref{mean+variance} \hfil \\
\phantom{M} & & & \\
\hfil$\V\{f(\tau_n)\}$\hfil & \hfil Theorem \ref{mean+variance} \hfil & \hfil Theorem \ref{mean+variance} \hfil & \hfil Theorem \ref{mean+variance} \hfil \\
\phantom{M} & & & \\
         \hline 
    \end{tabular}
    \smallskip
    \caption{Summary of the main asymptotic results under three models. $\tau_n$ refers to a random tree with $n$ leaves under the model, and $f(\tau_n)$ is the associated random CP rank. Properties of random trees are the same for uniformly random leaf-labeled unordered trees and for uniformly random unlabeled ordered trees.}
    \label{table:3}
\end{table}

\clearpage
\begin{figure}[tb]
\centering
\includegraphics[width=0.85\textwidth]{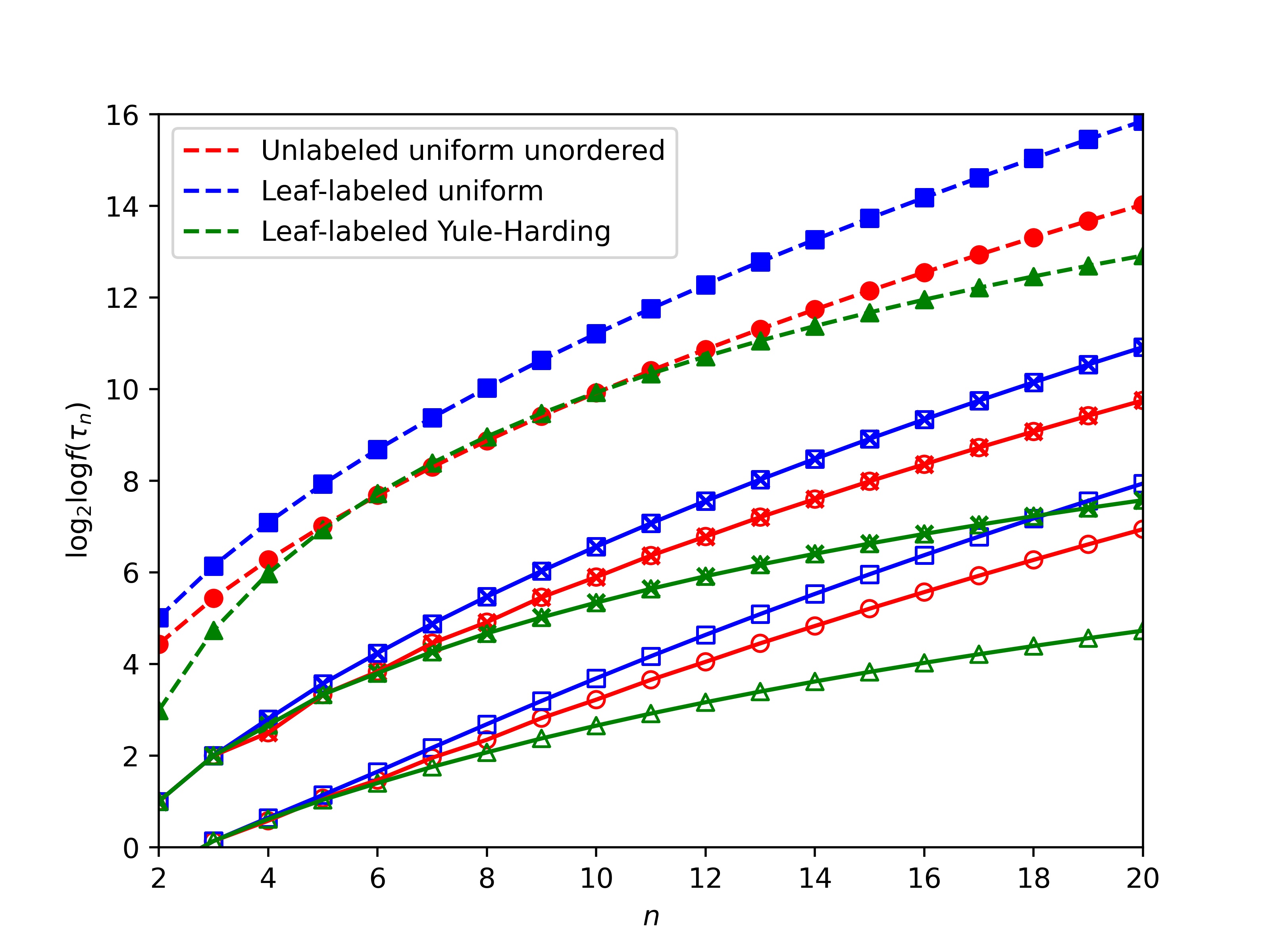}
\caption{Expected value of the double logarithm of CP rank, $\mathbb{E} \{ \log_2 \log f(\tau_n) \}$, under three models, for $n=2$ to 20: uniformly random unlabeled unordered binary trees, uniformly random leaf-labeled binary trees, and Yule--Harding leaf-labeled binary trees. Exact values of $\mathbb{E} \{ \log_2 \log f(\tau_n) \}$ (open symbols) appear alongside exact values of the expected tree height $\mathbb{E} \{ H_n \}$ (open symbols superimposed with crosses) under the three models and the asymptotic expressions (closed symbols, dashed lines): $\kappa \sqrt{n}$ for unlabeled uniform unordered (Theorem \ref{uniformtree}ii), $2 \sqrt{\pi n}$ for leaf-labeled uniform (Theorem \ref{cladogram}), and $\alpha \log n$ for leaf-labeled Yule--Harding (Theorem \ref{rbst}). $\kappa \approx 3.13699$, $\alpha \approx 4.31107$.}
\label{fig:1}
\end{figure}
\clearpage
\begin{figure}[tb]
\centering
\includegraphics[width=0.85\textwidth]{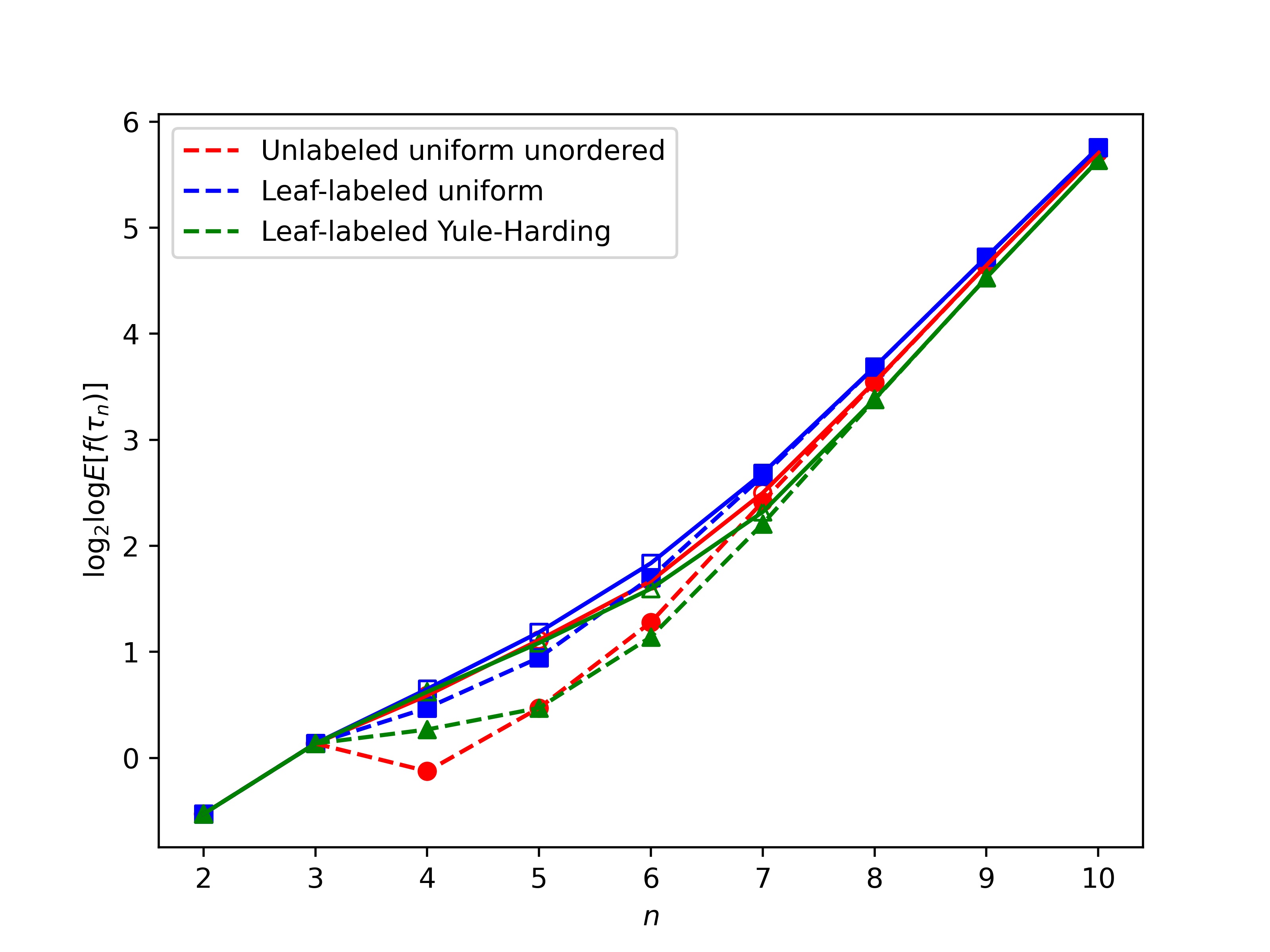}
\caption{Expected value of the CP rank, $\mathbb{E} \{ f(\tau_n) \}$, under three models, for $n=2$ to 10: uniformly random unlabeled unordered binary trees, uniformly random leaf-labeled binary trees, and Yule--Harding leaf-labeled binary trees. Exact values of $\log_2 \log \mathbb{E} \{ f(\tau_n) \}$ (open symbols) appear alongside asymptotic expressions $\log_2 \log (\pi_n c_{n-1})$ from Theorem \ref{mean+variance} (closed symbols, dashed lines), where $\pi_n$ follows \eqref{eq:piCatalan} for leaf-labeled uniform and \eqref{eq:pirbst} for leaf-labeled Yule--Harding and $c_{n-1}$ is the CP rank of the caterpillar with $n-1$ internal nodes and $n$ leaves \eqref{eq:recursion-c}. For unlabeled uniform unordered, $\pi_n$ is computed as the exact $1/U_n$, where $U_n$ is the Wedderburn--Etherington number \eqref{eq:Wedderburn}.}
\label{fig:2}
\end{figure}
\clearpage
\begin{figure}[tb]
\centering
\includegraphics[width=0.85\textwidth]{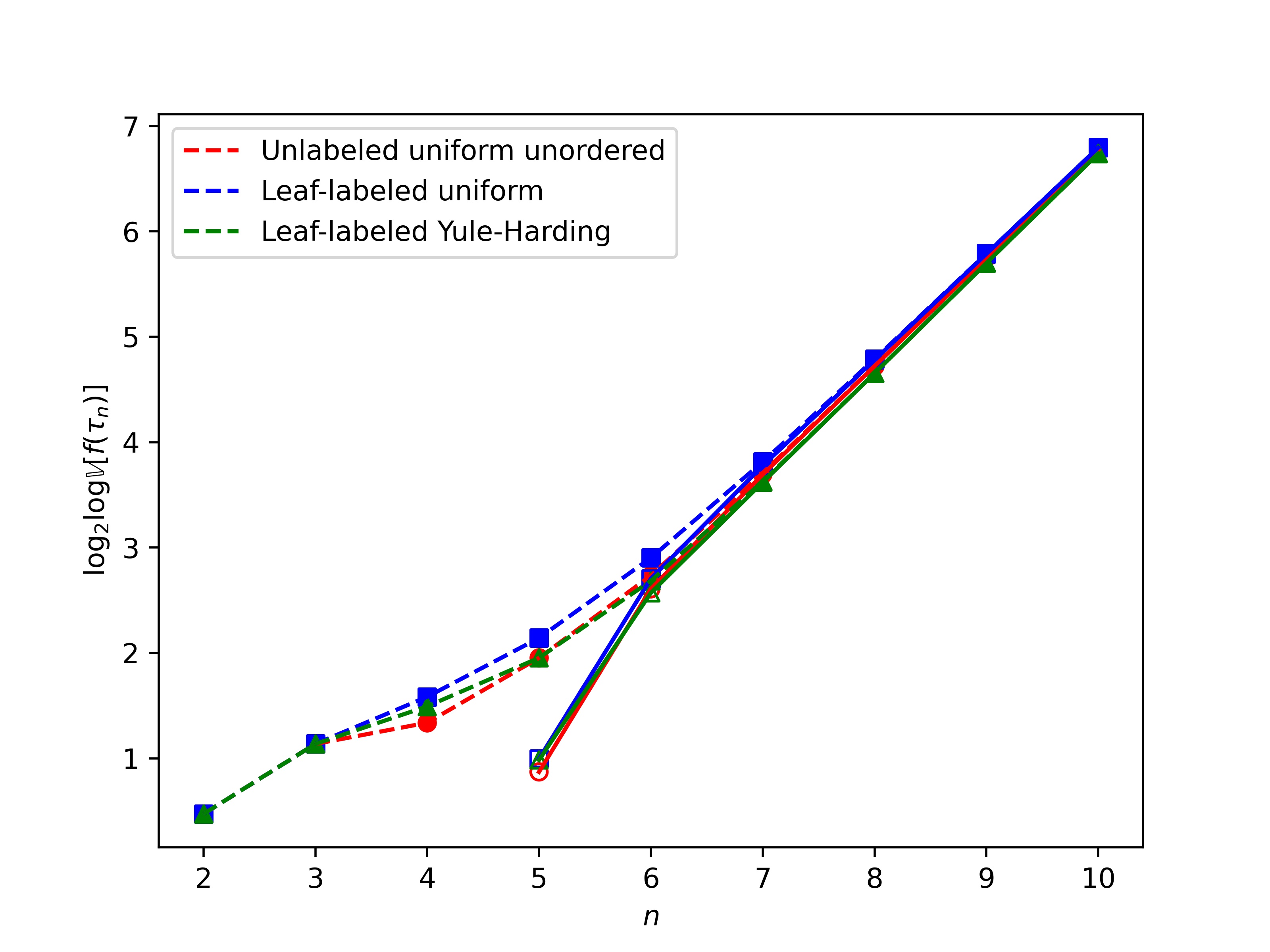}
\caption{Variance of the CP rank, $\mathbb{E} \{ f(\tau_n) \}$, under three models, for $n=2$ to 10: uniformly random unlabeled unordered binary trees, uniformly random leaf-labeled binary trees, and Yule--Harding leaf-labeled binary trees. Exact values of $\log_2 \log \mathbb{V} \{ f(\tau_n) \}$ (open symbols) appear alongside  asymptotic expressions $\log_2 \log (\pi_n c_{n-1}^2)$ from Theorem \ref{mean+variance} (closed symbols, dashed lines), where $\pi_n$ follows \eqref{eq:piCatalan} for leaf-labeled uniform and \eqref{eq:pirbst} for leaf-labeled Yule--Harding and $c_{n-1}$ is the CP rank of the caterpillar with $n-1$ internal nodes and $n$ leaves \eqref{eq:recursion-c}. For unlabeled uniform unordered, $\pi_n$ is computed using the exact $1/U_n$, where $U_n$ is the Wedderburn--Etherington number \eqref{eq:Wedderburn}.}
\label{fig:3}
\end{figure}

\clearpage

\end{document}